\RequirePackage{fix-cm}
\documentclass[smallextended,envcountsect]{svjour3}      
\smartqed  
\usepackage{graphicx}
\usepackage{amsfonts,amsmath,amssymb,amscd}

\usepackage{soul,color}
\sethlcolor{yellow}

\newcommand{\ignore}[1]{}

\newcommand{\e}{\mathbb{E \,}}

\newcommand{\BR}{\mathbb{R}}
\newcommand{\p}{\mathbb{P}}
\newcommand{\var}{{\rm Var}}

\newcommand{\C}{{\log M}}

\newcommand{\X}{{\mathcal X}}

\journalname{Probability Theory and Related Fields}
\begin{document}

\title{Bounded size bias coupling: a gamma function bound, and universal Dickman-function behavior
}

\titlerunning{Bounded size bias coupling}        

\author{Richard Arratia         \and
        Peter Baxendale 
}


\institute{R. Arratia \at
              Department of Mathematics, University of Southern California,
Los Angeles, CA 90089 \\
              \email{rarratia@usc.edu}           
           \and
           P. Baxendale \at
              Department of Mathematics, University of Southern California,
Los Angeles, CA 90089 \\
                \email{baxendal@usc.edu}
}

\date{Received: March 26, 2014 / Accepted: June 20, 2014 }

\maketitle

\begin{abstract}
Under the assumption that the distribution of a nonnegative random variable $X$ admits a bounded coupling with its
size biased version,  we prove  simple and strong concentration bounds.
In particular the upper tail probability is shown to  decay at least as fast as
the reciprocal of a Gamma function,
 guaranteeing a moment generating function that converges everywhere.  The class of infinitely divisible distributions with finite mean, whose L\'evy measure is supported on an interval contained in $[0,c]$ for some $c < \infty$, forms a special case in which this upper bound is logarithmically sharp. In particular the asymptotic estimate for the Dickman function, that $\rho(u) \approx u^{-u}$ for large $u$, is shown to be universal for this class.

A special case of our bounds arises when $X$ is a sum of independent random variables, each admitting a 1-bounded size bias coupling.  In this case, our bounds are comparable to Chernoff--Hoeffding bounds; however,
ours are broader in scope, sharper for the upper tail, and equal for the lower tail.

We discuss \emph{bounded} and \emph{monotone} couplings, give a sandwich principle, and show how this gives an easy conceptual proof that any finite positive mean
sum of independent Bernoulli random variables
admits a 1-bounded coupling with the same
conditioned to be nonzero.

\end{abstract}

\section{Introduction and Main Results} \label{sect intro}

For any nonnegative random variable $X$
with $0 < \e X < \infty$, we say that the distribution of $Y$ is the size biased distribution of $X$, written $Y =^d X^*$, if
the Radon-Nikodym derivative  of the distribution of $Y$, with respect to the distribution of $X$, is given by  \mbox{$\p(Y \in dx)/\p(X \in dx)$} $ = x/ \e X$. If $Y =^d X^*$, then for all
bounded measurable $g$,  $\e g(Y) = \e( X g(X)) / \e X$.  For much more information about size biased distributions see \cite{AGK}, or \cite[pp 78--80]{ABT}.

In this paper we shall assume that $X$ admits a $c$-bounded size bias coupling.  More precisely we make the hypothesis

\vspace{1ex}

\noindent {\bf BSBC}:   $X$ is a non-negative random variable with positive finite expected value $\e X = a$, and for $Y =^d X^*$ there exists a coupling in which $  Y \le X+c$
for some $c \in (0,\infty)$.

\vspace{2ex}

Throughout the paper the numbers $a$ and $c$ will always refer to their definitions in {\bf BSBC}.  Examples of random variables which admit a $c$-bounded size bias coupling are given in Section \ref{sect examples}, and some equivalent formulations of the assumption {\bf BSBC} are given in Section \ref{sect bounded}

Define
$$
  F(x ) := \p( X \le x), \ \ G(x) := \p(X \ge x).
$$
When $X$ has finite mean $a$, concentration inequalities refer to upper bounds on the upper tail probability $G(x) = \p(X \ge x)$ for $x \ge a$ and on the lower tail probability $F(x) = \p(X \le x)$ for $x \le a$.  For a review of early results on concentration inequalities see Ledoux \cite{ledoux}.  More recently Chatterjee \cite{chatterjee} has used Stein's method for exchangeable pairs to obtain concentration inequalities, see also \cite{chatterjeedey}.

The remarkably effective idea of using bounded size bias couplings to prove concentration inequalities comes from
Ghosh and Goldstein \cite{GhoshGold};  their proof is inspired by the $x \mapsto e^x$ is convex argument used to prove the Hoeffding concentration bounds, see \cite{hoeffding}.  For many examples of the application of concentration bounds derived from size bias couplings, in situations involving \emph{dependence}, see \cite{GhoshGold2,GhoshGold,AGKillian,BGI}.  Some details of the example in \cite{AGKillian} are given in Section \ref{sect examples}.  An extension to a multivariate setting is given in Ghosh and I\c{s}lak \cite{GI}

Number theorists denote by
   $
   \Psi(x,y)
   $
the number of positive integers not exceeding $x$ and free of prime factors larger than $y$.  Dickman \cite{dickman} showed that for any $u > 0$
   $$
   \lim_{y \to \infty} \Psi(y^u,y) y^{-u} = \rho(u)
   $$
where $\rho(u)$, the Dickman function, is the unique continuous function satisfying
    $$
    \left\{ \begin{array}{rcll}
              u\rho'(u) & = & -\rho(u-1) & \quad \mbox{if }u > 1 \\
              \rho(u)  & =  & 1 & \quad  \mbox{if }0 \le u \le 1.
              \end{array} \right.
              $$
In Hildebrand and Tenenbaum \cite[Lemma 2.5]{HT93}
it is shown that
       $$ 
         u \rho(u) = \int_{u-1}^u \rho(v)\,dv, \quad  u \ge 1
       $$ 
and then a simple inductive proof gives the inequality
         \begin{equation} \label{dickman ineq}
         0 < \rho(u) \le \frac{1}{\Gamma(u+1)}, \quad u \ge 0,
         \end{equation}
where $\Gamma(u)$ denotes the usual Gamma function.  It is well known that $\log \Gamma(u) \sim u \log u$ as $u \to \infty$, and the Dickman function $\rho(u)$ exhibits similar asymptotic behavior: $\log \rho(u) \sim -u \log u$ as $u \to \infty$, see \cite[Cor 2.3]{HT93}.

Inspired by the proof of \eqref{dickman ineq}, we prove stronger concentration inequalities than \cite{GhoshGold}, under weaker hypotheses, and with a simpler proof.  In the following, the notation $\lfloor x \rfloor$ is used for the floor function, that is, the greatest integer less than or equal to $x$.

\begin{theorem}\label{simple thm}
Assume {\bf BSBC}.  Given $x$ let
\begin{equation}\label{def k}
   k  = \lfloor \frac{| x-a|}{c} \rfloor,
\end{equation}
so that $k$ is a nonnegative integer, possibly zero.
Then
\begin{equation}\label{upper}
\text{ for } x\ge a, \ \      G(x) \le u(x,a,c) := \prod_{0 \le i \le k} \frac{a}{x-ic},
\end{equation}
and
\begin{equation}\label{lower}
\text{ for } 0 \le x \le a, \ \    F(x) \le \ell(x,a,c) := \prod_{0 < i \le k} \frac{x+ic}{a}.
\end{equation}
\end{theorem}

\begin{remark} \label{remark scaling} {\bf Scaling.} It is simplest, both for notation and concept, to work with the special case where the constant $c$ in {\bf BSBC} satisfies $c=1$.  The results derived for this special case easily transform into results for the general case, since if $Y =^d X^*$ and $Y \le X+c$, then
\begin{equation}\label{make c one with everything}
(Y/c) =^d (X/c)^*, \ \text{ and }(Y/c) \le (X/c) + 1.
\end{equation}
In particular the upper bounds in \eqref{upper} and \eqref{lower} satisfy
\begin{equation}\label{scaling}
     u(x,a,c) = u(x/c,a/c,1), \ \ \ \ell(x,a,c) = \ell(x/c,a/c,1).
\end{equation}
for all $a,c>0$ and $x \ge 0$.
\end{remark}

An opportunity to use \eqref{scaling} occurs in the following result, which provides a more convenient closed form version of the concentration inequalities above.

\begin{theorem}\label{theorem gamma}  The upper tail bound defined by \eqref{upper} and the lower tail bound defined by \eqref{lower}
satisfy
 \begin{equation}\label{gamma upper 1}
   \text{ for } x\ge a, \ \     u(x,a,1) \le \frac{ a^{x-a} \ \Gamma(a+1)}{\Gamma(x+1)},
\end{equation}
and
\begin{equation}\label{gamma lower 1}
  \text{ for } 0 \le x \le a, \ \    \ell(x,a,1) \le  \frac{\Gamma(a+1)}{a^{a-x}\Gamma(x+1)},
\end{equation}
with, in each case, equality if and only if $x-a$ is an integer.
\end{theorem}



It is an immediate consequence of these results that if $X$ satisfies {\bf  BSBC} then
     \begin{equation} \label{log G}
     \limsup_{x \to \infty} \frac{\log G(x)}{(x/c) \log x} \le -1
     \end{equation}
In Section \ref{sect inf div} we present a class of random variables satisfying {\bf BSBC} for which the $\limsup$ and inequality in \eqref{log G} can be replaced by $\lim$ and equality, see Theorem \ref{thm universal}.

An alternative approach, described in Section \ref{sect AB}, involves an upper bound on the moment generating function of $X$ and gives the following result.

\begin{theorem}  \label{theorem AB}
Assume {\bf BSBC}.  Then
  \begin{equation} \label{AB upper}
   \text{ for } x\ge a, \ \       G(x)  \le \left(\frac{a}{x}\right)^{x/c} e^{(x-a)/c},
    \end{equation}
and
    \begin{equation} \label{AB lower}
   \text{ for } 0 \le x \le a, \ \    F(x) \le \left(\frac{a}{x}\right)^{x/c} e^{(x-a)/c}.
    \end{equation}
 \end{theorem}

Here we use the convention that
$(1/0)^0  = \lim_{x \to 0^+} (1/x)^x = 1$, so we can regard $\displaystyle{ \left(a/x\right)^{x/c}}$ as being well-defined and taking the value 1 when $x = 0$. Some elementary analysis, see Lemmas \ref{lemma hoeffding upper} and \ref{lemma hoeffding lower}, allows us to replace the bounds in Theorem \ref{theorem AB} with the strictly weaker (for $x \neq a$) bounds obtained by Ghosh and Goldstein in \cite{GhoshGold}.

\begin{corollary}  \label{cor-intro}  Assume {\bf BSBC}.  Then
   \begin{equation} \label{G-intro}
   \text{ for } x\ge a, \ \    G(x) \le \exp\left(-\frac{(x-a)^2}{c(a+x)}\right),
   \end{equation}
and
    \begin{equation} \label{F-intro}
   \text{ for } 0 \le x \le a, \ \   F(x)  \le \exp\left(-\frac{(x-a)^2}{2ca}\right).
   \end{equation}
 \end{corollary}

\vspace{2ex}

For a simple application of these bounds, suppose that $X = X_1+X_2+ \cdots + X_n$ is the sum of independent random variables with values in the interval $[0,c]$.  Then $X$ satisfies {\bf BSBC}, see Section \ref{sect examples}.  The estimates in Theorem \ref{theorem AB} are the standard Hoeffding inequalities, see Hoeffding \cite{hoeffding} and Section \ref{sect Hoeffding}, and the estimates in Corollary \ref{cor-intro} are those obtained by Chatterjee \cite{chatterjee} as an simple example of his results on concentration inequalities for exchangeable pairs.

\vspace{2ex}

The proofs of Theorems \ref{simple thm}, \ref{theorem gamma} and \ref{theorem AB} and Corollary \ref{cor-intro} are given in Sections \ref{sect product}, \ref{sect gamma} and \ref{sect AB}.  Some simple examples of random variables satisfying {\bf BSBC} are given in Section \ref{sect examples}, and a particular family of infinitely divisible measures satisfying {\bf BSBC}, the so called L\'evy($[0,c]$) random variables, is studied in Section \ref{sect inf div}.  Section \ref{sect bounded} deals with the relationship between monotone and bounded couplings.  Finally in Sections \ref{sect upper and lower} and \ref{sect comparison} we consider the relative strengths of the bounds in Theorems \ref{simple thm} and \ref{theorem gamma} (coming from the elementary argument in Lemma \ref{lemma one step}) as opposed to those in Theorem \ref{theorem AB} and Corollary \ref{cor-intro} (obtained via the moment generating function); the technical calculations appear in Section \ref{sect upper and lower} and the discussion is in Section \ref{sect comparison}. 

 \section{Product bounds} \label{sect product}

We start with an elementary argument giving powerful bounds on the upper and lower tail probabilities.

\begin{lemma}\label{lemma one step}
Assume {\bf BSBC}.
Then
\begin{equation}\label{upper one step}
\forall x>0, \ \      G(x) \le \frac{a}{x} \, G(x-c),
\end{equation}
and
\begin{equation}\label{lower one step}
  \forall x, \ \ F(x) \le \frac{x+c}{a} \, F(x+c).
\end{equation}
\end{lemma}
\begin{proof}
To prove the upper bound on $G(x)$, note that {\bf BSBC} implies that the event $Y \ge x$ is a subset of the event $X \ge x-c$.  Hence for $x>0$,
\begin{eqnarray*}
x G(x)  = x \, \e 1(X \ge x) & \le & \e(X 1(X \ge x)) \\
 &= & a \, \p(Y \ge x)
 \\
 &\le& a \,G(x-c).
\end{eqnarray*}
When $x>0$ we can divide by $x$ to get \eqref{upper one step}.

To prove the upper bound on $F(x)$, note that {\bf BSBC} implies that the event $Y \le x$ is a superset of the event $X \le x-c$.  Hence
\begin{eqnarray*}
x F(x)  = x \, \e 1(X \le x) & \ge & \e(X 1(X \le x)) \\
 &= & a \, \p(Y \le x)
 \\
 &\ge& a \,F(x-c).
\end{eqnarray*}
This does not require that $x$  be positive;  for $x<0$ it is the trivial inequality, that $0 \ge 0$.  Replacing $x$ by $x+c$ and 
dividing by $a>0$ yields \eqref{lower one step}. \qed
\end{proof}

\vspace{2ex}

\noindent {\bf Proof of Theorem \ref{simple thm}.}   Given $x>0$, the obvious strategy for obtaining good bounds is to iterate \eqref{upper one step} or
\eqref{lower one step} for as long as the new value of $x$, say $x'=x \pm i c$, still gives a favorable ratio,
$a/x'$ in \eqref{upper one step}, or $(x'+c)/a$ in \eqref{lower one step}, and using $G(t) \le 1$ or $F(t) \le 1$ as needed, to finish off.  The proof is now a simple matter of carrying out this strategy. \qed

\begin{corollary}\label{cor mgf}  Assume {\bf BSBC}.  Then the moment generating function of $X$ is finite everywhere, that is, for all $\beta \in \BR$, $M(\beta) := \e e^{\beta X} < \infty$.
\end{corollary}
\begin{proof}
Since $X \ge 0$, trivially $M(\beta) \le 1$ if $\beta \le 0$, so assume  $\beta > 0$.  For motivation:  as $x$ increases by $c$, $e^{\beta  x}$ increases by a factor of $e^{\beta c}$, while the upper bound $u(x,a,c)$ on $G(x)$ decreases by a factor of
$a/(x+c)$;  hence for $x > x_0 := 2 a e^{\beta c}$, the product $e^{\beta  x} u(x,a,c)$ decreases by a factor of at least 2.  More precisely, for $x \ge a$ we have $u(x+c,a,c) = \frac{a}{x+c}u(x,a,c)$ and so
   $$
   e^{\beta(x+c)}u(x+c,a,c) = \left(\frac{ae^{\beta c}}{x+c}\right) e^{\beta x}u(x,a,c) \le \frac{1}{2} e^{\beta x} u(x,a,c)
   $$
for $x > x_0$.   Writing
   \begin{eqnarray*}
    M(\beta) & = & \mathbb{E}\left[ e^{\beta X}1(X < x_0)\right] + \sum_{i \ge 0} \mathbb{E} \left[e^{\beta X}1(x_0+ic \le X < x_0+ic+c)\right]\\
      & \le & \mathbb{E}\left[ e^{\beta X}1(X < x_0)\right] + \sum_{i \ge 0} e^{\beta(x_0+ic+c)} u(x_0+ic,a,c),
  \end{eqnarray*}
the series on the right side is bounded by a geometric series with ratio 1/2, and the net result will be
$M(\beta) \le  \exp(\beta x_0) + 2 \exp(\beta(x_0+c))$.
\qed \end{proof}

\begin{remark} \label{rem include}  Note the difference between the indexing in the products (\ref{upper},\,\ref{lower}): $u(x,a,c)$ \emph{includes} the factor indexed by $i=0$, while $\ell(x,a,c)$ \emph{excludes} the factor indexed by $i=0$.  In case  $x \in (a,a+c)$, which is equivalent to $a < x$ and $k=0$, the bound in \eqref{upper} has one factor, and simplifies to $G(x) \le u(x,a,c) = a/x < 1$.  In case $x \in (a-c,a)$, which is equivalent to $x < a$ and $k=0$, the bound in \eqref{lower} has no factors, and simplifies to the trivial observation $F(x) \le \ell(x,a,c)=1$.
\end{remark}

\begin{remark}  \label{rem ell j} {\bf Product bound combined with the one-sided Chebyshev inequality.}
The recursive nature of Lemma \ref{lemma one step} allows the possibility of combining it with other information about $F(x)$ or $G(x)$.  Here we pursue 
one possibility.

 For a random variable $X$ with mean $a$ and variance $\sigma^2$,  the one-sided Chebybshev inequality states that, for all $x \le a$,
$F(x) = \p(X \le x)$ 
is upper bounded by $\sigma^2/(\sigma^2+(a-x)^2)$.
In our situation, $Y =^d X^*$ and $\e X=a$ yields $\e X^2 = a \, \e Y$, and   $Y \le X+c$  implies $\e Y \le a+c$. Hence
$\e X^2 = a \, \e Y \le a(a+c)$, so that $\sigma^2 \le ac$.    Thus, under the hypotheses of Lemma \ref{lemma one step}, for all $x \le a$,
\begin{equation}\label{one sided}
   \var X \le ac,  \ \    F(x) \le \frac{ac}{ac+(a-x)^2}.
\end{equation}
We want to improve on \eqref{lower} by using one-sided
Chebyshev in combination with iteration of \eqref{lower one step}.  More precisely, given $x < a$ and any non-negative integer $j$ such that $x+jc \le a$ we can iterate \eqref{lower one step} from $x$ to $x+jc$ and then use the one-sided Chebyshev inequality at $x+jc$ to obtain
    \begin{equation} \label{lower j}
    F(x) \le \ell_j(x,a,c):=\frac{ac}{ac+(a-x-jc)^2} \prod_{0 < i \le j}\frac{x+ic}{a}.
     \end{equation}

 \end{remark}

\section{Gamma bounds}  \label{sect gamma}

Here we replace the product bounds $u(x,a,c)$ and $\ell(x,a,c)$ in Theorem \ref{simple thm} by the simpler (but weaker) bounds given in Theorem \ref{theorem gamma}.  In this section we restrict to the case $c = 1$.  Results for general $c > 0$ can be recovered using equation \eqref{scaling}, see Remark \ref{remark scaling}.  For $c=1$, making use of $z\Gamma(z) = \Gamma(z+1)$, the conclusions \eqref{upper} and \eqref{lower} in Theorem \ref{simple thm} can be rewritten as: for $x \ge a$ and $k=\lfloor x-a \rfloor$,
\begin{equation}\label{upper c 1}
    G(x) \le u(x,a,1) := \prod_{0 \le i \le k} \frac{a}{x-i} =  \frac{a^{k+1} \Gamma(x-k)}{\Gamma(x+1)},
\end{equation}
and for $0 \le x \le a$ and $k=\lfloor a-x \rfloor$,
\begin{equation}\label{lower c 1}
F(x) \le \ell(x,a,1) := \prod_{0 < i \le k} \frac{x+i}{a} =  \frac{\Gamma(x+k+1)}{a^k\Gamma(x+1)}.
\end{equation}
These upper and lower tail bounds might be viewed as too complicated;  as $x$ varies
they are not closed-form expressions, and are not analytic functions.  Here we replace them by expressions which are analytic in $a$ and $x$.

\begin{lemma} \label{lemma gamma}
 For $a > 0$ and $0 \le f \le 1$
\begin{equation}\label{gamma inequality}
     a^{1-f}\Gamma(a+f) \le \Gamma(a+1),
\end{equation}
with equality for $f=0,1$ and strict inequality for $f \in (0,1)$.
\end{lemma}

\begin{proof}  The result is true (with equality) when $f = 0$ or $1$, so we can assume $0 < f < 1$.  We use the integral formula $\Gamma(x) = \int_0^\infty t^{x-1}e^{-t}dt$ for $x > 0$.   Writing $t^{a+f-1} = (t^{a-1})^{1-f} (t^a)^f$ and using H\"older's inequality (with $p = 1/(1-f)$) gives
   $$\Gamma(a+f)  < \left[\Gamma(a)\right]^{1-f}\left[\Gamma(a+1)\right]^f.
   $$
Since $ a \Gamma(a) = \Gamma(a+1)$ we get
    $$a^{1-f}\Gamma(a+f) < \left[a\Gamma(a)\right]^{1-f}\left[\Gamma(a+1)\right]^f = \Gamma[a+1],
   $$
and we are done.  \qed \end{proof}

\vspace{2ex}

\noindent {\bf Proof of Theorem \ref{theorem gamma}.}
(i) Let $k = \lfloor x-a \rfloor$ and $f= x-a-k \in [0,1)$.  Since $c=1$ and $x \ge a$, this is consistent with the notation in Theorem \ref{simple thm}.  Note that $x-k=a-f$, and combine \eqref{upper c 1} with  \eqref{gamma inequality}.

\noindent (ii)  Let $k = \lfloor a-x \rfloor$ and $f = a-x-k \in [0,1).$  Since $c=1$ and $a \ge x$, this is consistent with the notation in Theorem \ref{simple thm}.  Replacing $f$ by $1-f$ in \eqref{gamma inequality} gives $a^f \Gamma(a-f+1) \le \Gamma(a+1)$.  Combining this with \eqref{lower c 1} and noting that $x+k=a-f$ gives the result.  \qed

\section{Moment generating function} \label{sect AB}

We recall our notation, $M(\beta) := \e e^{\beta X}$ for the moment generating function of $X$.  We observe that if $X$ is Poisson with parameter $a$, then $\C(\beta) = a(e^{\beta}-1)$, and $X$ admits a $c$-bounded size bias coupling with $c=1$, so that in this case, the inequality
\eqref{mgf} holds as an equality for all $\beta$.

\begin{proposition} \label{prop mgf}
Assume {\bf BSBC}.  Then the moment generating function $M(\beta)$ for $X$ satisfies
    \begin{equation} \label{mgf}
    \C(\beta) \le  \frac{a}{c}\left(e^{\beta c}-1\right)
    \end{equation}
for all $\beta \in \BR$.
    \end{proposition}

\begin{proof}  We know from Corollary \ref{cor mgf} that the moment generating function $M(\beta)$ is finite for all $\beta \in\BR$.  It follows that $M$ is continuously differentiable and $ M'(\beta) = \e(Xe^{\beta X})$.  Moreover, since $Y$ is the size biased version of $X$, we have $\e(Xe^{\beta X})  = a \e(e^{\beta Y})$.  Together we have
    $
    M'(\beta)  = a \e(e^{\beta Y}).
    $

For $\beta  \ge 0$ we have $e^{\beta Y} \le e^{\beta (X+c)} = e^{\beta c}e^{\beta X}$ so that
       $$
       M'(\beta ) = a \e(e^{\beta Y}) \le ae^{\beta c} \e(e^{\beta X}) = a e^{\beta c} M(\beta ).
       $$
Then
      $$
     (\log M)'(\beta ) \le a e^{\beta c}
      $$
so that
$$ 
   \C(\beta) = \C(\beta)-\C(0)  \le \int_0^\beta  a e^{uc}\,du = \frac{a}{c}\left(e^{\beta c}-1\right).
$$ 
  for all $\beta \ge 0$.

For $\beta  \le 0$ we have $e^{\beta Y} \ge e^{\beta (X+c)} = e^{\beta c}e^{\beta X}$ so that
       $$
       M'(\beta ) = a \e(e^{\beta Y}) \ge ae^{\beta c} \e(e^{\beta X}) = a e^{\beta c} M(\beta ).
       $$
Then
      $$
      (\C)'(\beta ) \ge a e^{\beta c}
      $$
so that
    $$
  -\C(\beta) = \C(0)  - \C(\beta ) \ge \int_\beta ^0 a e^{uc}\,du = \frac{a}{c}\left(1-e^{\beta c}\right),
    $$
  and the proof is complete.
     \qed \end{proof}

The upper tail and lower tail bounds in Theorem \ref{theorem AB} can now be obtained using the standard ``large deviation upper bound'' method together with the information about $M(\beta)$ in Proposition \ref{prop mgf}.  Here are the details.

 \vspace{2ex}

\noindent {\bf Proof of Theorem \ref{theorem AB}.}  Suppose first $x \ge a$.  For any $\beta  \ge 0$ we have
   $$
   \p (X \ge x) = \p(e^{\beta X} \ge e^{\beta x}) \le M(\beta )/e^{\beta x} \le \exp\left\{ \frac{a}{c}\left(e^{\beta c}-1 \right)-\beta x \right\}.
   $$
Choosing $\beta  = (1/c) \log(x/a) \ge 0$ we get
   $$
   \p(X \ge x) \le  \exp\left\{ \frac{a}{c}\left(\frac{x}{a}-1\right)-\frac{x}{c} \log\left(\frac{x}{a}\right)  \right\}
    = e^{(x-a)/c} \left(\frac{a}{x}\right)^{x/c}.
    $$
Now suppose $0 \le x \le a$.  For any $\beta  \le 0$ we have
   $$
   \p (X \le x) = \p(e^{\beta X} \ge e^{\beta x}) \le M(\beta )/e^{\beta x} \le \exp\left\{ \frac{a}{c}\left(e^{\beta c}-1 \right)-\beta x \right\}.
   $$
If $0 < x \le a$ take $\beta  = (1/c) \log(x/a) \le 0$ to get
   $$
   \p(X \le x) \le  \exp\left\{ \frac{a}{c}\left(\frac{x}{a}-1\right)-\frac{x}{c} \log\left(\frac{x}{a}\right)  \right\}
    = e^{(x-a)/c} \left(\frac{a}{x}\right)^{x/c},
    $$
while if $x = 0$ let $\beta \to -\infty$ to get $\p(X \le 0) \le e^{-a/c}$.
\qed

\vspace{2ex}

The passage from Theorem \ref{theorem AB} to Corollary \ref{cor-intro} is a direct application of the following two lemmas.

\begin{lemma} \label{lemma hoeffding upper} Suppose $0 < a \le x$ and $c >0$.  Then
     \begin{equation} \label{hoeffding upper tail GG}
    \left(\frac{a}{x}\right)^{x/c} e^{(x-a)/c} \le e^{-(x-a)^2/(ca+cx)},
    \end{equation}
with strict inequality whenever $x > a$.   \end{lemma}

\begin{proof} The expressions both take the value 1 when $x = a$, so we can assume $x > a$.  Jensen's inequality applied to the function $f(x) = 1/x$ on the interval $[1,1+u]$ gives $\log(1+ u) > 2u/(2+u)$ for $u > 0$.  Taking $u = (x-a)/a$ we get $\log(x/a) > 2(x-a)/(x+a)$ and so
   $$
   x \log(a/x) +x-a < \frac{-2x(x-a)}{x+a} + x-a = -\frac{(x-a)^2}{x+a}.
   $$
Dividing by $c$ and applying the exponential function to both sides gives \eqref{hoeffding upper tail GG}.
\qed \end{proof}

\begin{lemma}  \label{lemma hoeffding lower} Suppose $0 \le x \le a$ and $a,c >0$.  Then
     \begin{equation} \label{hoeffding lower tail GG}
    \left(\frac{a}{x}\right)^{x/c} e^{(x-a)/c} \le e^{-(x-a)^2/(2ca)},
    \end{equation}
with strict inequality whenever $x < a$.
    \end{lemma}

\begin{proof}  The expressions both take the value 1 when $x = a$, so we can assume $x < a$.  Integrating the inequality $x^{-1} \le (1+x^{-2})/2$ from $1$ to $v $ gives
$\log v < (v-1/v)/2$ for $v > 1$.  Putting $v=a/x$ gives
    $$
    x \log(a/x) -a+x \le \frac{x}{2}\left(\frac{a}{x}-\frac{x}{a}\right)-a+x = -\frac{(x-a)^2}{2a}.
    $$
Dividing by $c$ and applying the exponential function to both sides gives \eqref{hoeffding lower tail GG}.
\qed \end{proof}

\section{Examples admitting a 1-bounded size bias coupling} \label{sect examples}

Here we give some examples of random variables $X$ which satisfy {\bf BSBC} with $c = 1$.  Examples with general $c > 0$ can be obtained by scaling, see Remark \ref{remark scaling}.

\begin{example} \label{example poisson}  If $X$ has a Poisson distribution, it can be verified directly that $X+1 =^d X^*$.  Conversely if $X \ge 0$ with $0 < \e X <\infty$ and $X+1 =^d X^*$, then the inequality in the statement (and proof) of Proposition  \ref{prop mgf} can be replaced by equality, so that $X$ must have a Poisson distribution.

\end{example}

\begin{remark}[Sharpness of the bounds]
\label{remark sharp upper}
Suppose $X$ is Poisson distributed with mean $a \in(0,\infty)$.  Taking $x = 0$ in the lower tail bound \eqref{AB lower} from Theorem \ref{theorem AB} we get $F(0) \le e^{-a}$ whereas the exact value is $F(0) = e^{-a}$.  Therefore in this setting the lower tail bound \eqref{AB lower} from Theorem \ref{theorem AB} is sharp.

Now suppose further that $a \in (0,1]$.  When $x=n$, the upper tail bound \eqref{upper} in Theorem \ref{simple thm} simplifies, with $k=n-1$, to
$G(x) \le u(n,a,1)= a^n/n!$,
while for large $n$, $G(n) \sim \p(X=n) = e^{-a} a^n/n!$.  Hence, for large $x$ with $x$ an integer, the upper bound \eqref{upper} is sharp up to a factor of approximately $e^{a}$.  Letting $a \to 0$ so that $e^{a} \to 1$, one sees that the upper bound \eqref{upper}, for large $x$, is sharp up to a factor arbitrarily close to 1.
\end{remark}

\begin{example} \label{example inf div} L\'evy([0,1]), the infinitely divisible distributions with L\'evy measure supported on $[0,1]$.  This is the case $c=1$ of the L\'evy([0,$c$]) distributions, discussed in detail in Section \ref{sect inf div}.
\end{example}

\begin{example} \label{example 01}  Let $X$ be a random variable with values in $[0,1]$, with $\e X > 0$.  This includes, but is not restricted to, Bernoulli random variables.  The size biased version $Y$, say, of $X$ takes values in $[0,1]$ also, and hence $Y \le 1 \le 1+X$, so that $X$ admits a 1-bounded size bias coupling.
\end{example}

\begin{example} \label{example uniform}
The uniform distribution on $[0,b]$ admits a 1-bounded size bias coupling for any $b \in (0,4]$.
\end{example}
\begin{proof}
Suppose that $X$ is uniformly distributed on $[0,1]$, and that $Y$ is the size biased version of $X$.  On $[0,1]$,
the density of $X$ is $f_X(x)=1$, the density of $Y$ is $f_Y(x)=2x$, the cumulative distribution functions are  $F_X(t)=t$ and $F_Y(t)=t^2$, and the inverse cumulative distribution functions are $F_X^{(-1)}(u)=u$ and $F_Y^{(-1)}(u)=\sqrt{u}$ for $0 \le u \le 1$.  Observe that $\max_{0 \le u \le 1} (\sqrt{u}-u) = \frac{1}4$, achieved at $u=\frac{1}4$.  Hence, using the quantile transform as in the proof of Lemma \ref{lem-xy}, there is a $c$-bounded size bias coupling for the standard uniform with  $c=\frac{1}4$.  By scaling, as in Remark \ref{remark scaling}, the uniform distribution on $[0,b]$ admits a $b/4$-bounded size bias coupling.
\qed \end{proof}

\begin{proposition} \label{prop indep sum}  Suppose $X = \sum_{i} X_i$ is the sum of finitely or countably many independent non-negative random variables $X_i$ with $0 < \e X = \sum_i \e X_i <\infty$.  If each $X_i$ admits a 1-bounded size bias coupling, then so does $X$.
\end{proposition}
\begin{proof}  For each $i$ there is a coupling of $X_i$ with its size biased version $Y_i$, say, so that $Y_i \le X_i+1$.  Let $I$ denote an index $i$ chosen (independently of all the $X_j$ and $Y_j$) with probability $\e X_i/\e X$.  Then $Y:=X-X_I+Y_I$ is the size biased version of $X$, see \cite[Lemma 2.1]{GR96} or \cite[Sect. 2.4]{AGK}.  Since $Y_i \le X_i +1$ for all $i$, it follows that $Y \le X+1$.
\qed \end{proof}

\begin{example} \label{example binomial}  By Proposition \ref{prop indep sum}, if $X$ is Binomial, or more generally if $X$ is the sum of (possibly infinitely many) independent Bernoulli random variables (with possibly different parameters) with $0 < \e X < \infty$ then $X$ has a 1-bounded size bias coupling.
\end{example}

\begin{example} \label{example collisions}
Let $C(b,n)$ denote the number of collisions which occur when $b$ balls are tossed into $n$ equally likely bins.  (A collision is said to occur whenever a ball lands in a bin that already contains at least one ball.)  Clearly $C(b,n)$ can be expressed as the sum of dependent Bernoulli random variables.  In \cite[Prop 15]{AGKillian} it is shown that $C(b,n)$ admits a 1-bounded size bias coupling.

The paper \cite{AGKillian} studies the number $B(\kappa,n)$ of balls that must be tossed into $n$ equally likely bins in order to obtain $\kappa$ collisions. For the sake of concrete analysis, we suppose here that  $\kappa \sim n^\alpha$ for some fixed $\alpha \in (1/2,1)$;  in this case, the goal is to prove that the variance of $B(\kappa,n)$ is asymptotic to $n/2$.    R\'enyi's central limit theorem \cite{Renyi}
for the number of collisions $C(b,n)$ obtained with a given number $b$ of balls, combined with duality, leads easily to a central limit theorem for the number of balls $B(\kappa,n)$ that must be tossed;   the hard part, for getting the asymptotic variance, requires uniform integrability, to be obtained from concentration inequalities for $C(b,n)$.
The traditional tool, the Azuma-Hoeffding bounded difference martingale approach \cite{Azuma},  leads to
$\p( |C(b,n)- \e C(b,n) | \ge t) \le 2 \exp( -t^2/(2b))$.  In contrast, bounded size bias coupling inequalities such as Corollary \ref{cor-intro} involve $a = \e C(b,n)$ rather than $b$ in the denominator of the fraction in the exponential.  The situation in \cite{AGKillian} has $a/b \to 0$; it turns out that Azuma-Hoeffding inequality is inadequate in this setting, whereas the bounded size bias coupling inequalities are strong enough, see \cite[Sect 7]{AGKillian} for details.
  \end{example}

\section{L\'evy($[0,c]$), the infinitely divisible distributions with L\'evy measure supported on $[0,c]$} \label{sect inf div}

Among the distributions satisfying {\bf BSBC}
are those with characteristic function of the form
\begin{equation}\label{comp pois 1}
 \phi_X(u) := \e e^{i u X}  =\exp \left( a \left( i u \,  \alpha(\{0\}) +\  \int_{(0,c]} \frac{e^{iuy}-1}{y} \
   \alpha(dy) \right) \right)
\end{equation}
where $a \in (0,\infty)$ and  $\alpha$ is the probability distribution of a nonnegative nonzero random variable $D$, with $\p(D\in [0,c])=1$.  Given this characteristic function, the random variable $X$ has  $a=\e X$, and, with $X,D$ independent, $X^*=^d X+D$.  See \cite{AGK} for the connection between infinite divisibility of $X$ and size biasing of $X$ by adding an independent $D$.  Special cases include:
\begin{enumerate}
 \item $c=1, \p(D=1)=1; X$ is Poisson with mean $a$.

 \item $c=1,a=1, D$ is uniformly distributed on $(0,1)$; $X$ has density $f(x) = e^{-\gamma} \rho(x)$ where $\rho$ is Dickman's function and $\gamma$ is Euler's constant.

 \end{enumerate}
The characteristic function in \eqref{comp pois 1} can also be expressed as
 \begin{equation}\label{comp pois 2}
 \phi_X(u)  =
   \exp \left(   i a u \alpha_0+  \int_{(0,c]} \left(e^{iuy}-1 \right)  \gamma(dy) \right)
\end{equation}
Here $\gamma$ is a nonnegative measure on $(0,\infty)$, with $\gamma(dy) /\alpha(dy) = a/y$, and may be called  the L\'evy measure of the infinitely divisible random variable $X$;  see Sato \cite[Section 51]{sato}, Bertoin \cite[Section 1.2]{bertoin}, or \cite{AGK}.


\begin{theorem}\label{thm universal}
Suppose $X$ has distribution given by \eqref{comp pois 1} for $c >0$, and that for every $\varepsilon > 0$,
the probability measure $\alpha$ is
\emph{not} supported on $[0,c-\varepsilon]$.  Then $G(x) := \p(X \ge x)$ satisfies, as $x \to \infty$,
\begin{equation}\label{universal bound}
      G(x) \approx x^{-x/c}, \text{ that is}, \ \frac{\log G(x)}{(x/c) \log x} \to -1.
\end{equation}
\end{theorem}
\begin{proof}
As observed in the Introduction, the upper bound on $G(x)$ follows directly and easily from \eqref{upper} in Theorem \ref{simple thm} and \eqref{gamma upper 1} in Theorem \ref{theorem gamma}.

For the lower bound, let $\varepsilon>0$ be given, with $\varepsilon < c$.
Using the representation \eqref{comp pois 2}, the random variable $X$ can be realized as the constant $a \alpha_0$ plus the sum of the arrivals in the Poisson process $\X$ on $(0,c]$ with intensity measure $\gamma$.
Let $Z$ be the number of arrivals of $\X$ in $[c-\varepsilon,c]$.    We have $X \ge (c - \varepsilon) Z$.   This yields
$$
   G(x) = \p(X \ge x)  \ge  \p(Z \ge \frac{x}{c-\varepsilon}).
$$
Finally,
$Z$ is   a Poisson random variable with mean $\lambda = \gamma([c-\varepsilon,c] ) \ge (a/c) \alpha([c-\varepsilon,c]) > 0$. We recall an elementary calculation: for integers $k \to \infty$,  the Poisson ($\lambda$) distribution for $Z$ has $\p(Z \ge k) \ge  \p(Z=k)$ with
\begin{equation}\label{hide lambda}
  \log \p(Z=k) = -\lambda + k \log \lambda - \log k! \sim - \log k! \sim -k \log k.
\end{equation}
We use this with $k = \lceil x/(c-\varepsilon) \rceil \sim x/(c-\varepsilon)$ and $\log k \sim \log x$.
\qed \end{proof}

\begin{remark} Most probabilists are familiar with the $\approx$ notation of \eqref{universal bound},
with  $a_n \approx b_n$ defined to mean  $\log a_n \sim \log b_n$, for use in the context where $a_n$ grows or decays exponentially.  The standard example is the large deviation statement that
for i.i.d.\ sums, $\p (S_n \ge a n ) \approx \exp(-n I(a) )$, and the $\approx$ relation \emph{hides}  factors with a slower than exponential order of growth or decay, in this case $1/\sqrt{n}$.  When both $a_n$ and $b_n$ grow or decay even faster, for example with $a_n \sim n^{\pm n/c}$, an unfamiliar phenomenon arises, with $\approx$ hiding factors which grow or decay exponentially fast.  One example appears in the conclusion \eqref{universal bound} of Theorem \ref{thm universal}, where $x^{-x/c} \approx (x/c)^{-x/c} \approx 1/ \Gamma(x/c)$ as $x \to \infty$  --- with the last expression being relevant because  it corresponds to the Gamma function upper bound in Theorem \ref{theorem gamma}.  A second example occurs in
 \eqref{hide lambda}, where at first it appears strange that the parameter $\lambda$ does not appear on the right hand side  --- this reflects the fact that
for any fixed $\lambda, \lambda' >0$,  when $Z$ and $Z'$ are Poisson with parameters $\lambda, \lambda'$ respectively, $\p (Z=k) \approx \p (Z'=k)$ as $k \to \infty$.  The first example hides the exponentially growing factor  $c^{x/c}$,  and the second hides the factor $(\lambda' / \lambda)^k$.
\end{remark}

\section{Bounded coupling, monotone coupling, and a sandwich principle}\label{sect bounded}

Suppose that the distributions of random variables $X,Y$ have been specified, with cumulative distribution functions $F_X,F_Y$ respectively. We will clarify the relations
between hypotheses of the form
\begin{eqnarray}
\label{c bounded one side}  \exists \text{ a coupling},  & &  \p(Y \le X+c)=1, \\
\label{c bounded}      \exists \text{ a coupling}, & &  \p( |Y-X| \le c)=1,  \\
\label{c bounded monotone}     \exists \text{ a coupling},  & &  \p(Y \in [X,X+c])=1,  \\
\label{monotone}     \exists \text{ a coupling},   & &  \p( X \le Y)=1.
\end{eqnarray}
It is well-known that  if $Y =^d X^*$, then  \eqref{monotone} holds.   We observe that
if $Y =^d X^*$, then \eqref{c bounded one side} is the hypothesis {\bf BSBC}, while \eqref{c bounded} and \eqref{c bounded monotone} are used as 
hypotheses for the first results on 
concentration via size bias couplings, in
\cite{GhoshGold}.

\begin{proposition} \label{prop-equiv} Given that \eqref{monotone} holds, all of
\eqref{c bounded one side} -- \eqref{c bounded monotone} are equivalent.
\end{proposition}

The proof of Proposition \ref{prop-equiv} will be given later in this section.

\subsection{One sided bounds}
Among the four hypotheses \eqref{c bounded one side} -- \eqref{monotone}, only \eqref{monotone} is standard.  The relation of stochastic domination, that $X$ lies below $Y$, written $X \preceq Y$, is usually \emph{defined} by the condition
\begin{equation}\label{domination}
\forall t,  F_X(t) \ge F_Y(t),
\end{equation}
and it is well known that \eqref{domination} is equivalent to \eqref{monotone}.  See for example Lemma \ref{lem-xy}.  Stochastic domination is often considered in the more general context where
$X,Y$ are random elements in a \emph{partially ordered} set, see for example \cite{Liggett}.

\subsection{Two sided bounds: historical perspective}

Dudley \cite[Prop 1 and Thm 2]{Dudley68}, following earlier work of Strassen \cite{Strassen65}, proved the following result.  Given distributions $\mu$ and $\nu$ respectively for random variables $X$ and $Y$ taking values in a Polish metric space $(S,d)$ and $\alpha > 0$, write $A^\alpha = \{ x :  d(x,A) \le \alpha \}$  to denote the closed $\alpha$ neighborhood of the set $A \subset S$.  Then the following are equivalent:
\begin{eqnarray} \label{dudley 1}
& & \text{For all closed sets } A,  \ \ \mu(A) \le \nu(A^\alpha)   \\[1ex]
\label{dudley 1a}
&  &\text{For all closed sets } A,  \ \ \nu(A) \le \mu(A^\alpha)  \\[1ex]
& & \label{dudley 2}
\text{There exists a coupling under which } \p(d(X,Y) \le \alpha) = 1.
\end{eqnarray}
When  $X,Y$ are real valued random variables, and $\alpha = c$, condition \eqref{c bounded} is trivially equivalent to  \eqref{dudley 2}, and \eqref{dudley 1} or \eqref{dudley 1a} can be shown to be equivalent to 
$\forall t, \ F_Y(t-c) \le F_X(t) \le F_Y(t+c)$.

\subsection{Proof of Proposition \ref{prop-equiv}}

 It is clear that \eqref{c bounded monotone} implies \eqref{c bounded} implies \eqref{c bounded one side}.  It remains to show that \eqref{monotone} and \eqref{c bounded one side} together imply \eqref{c bounded monotone}, and this follows immediately by using the implication i) implies iii) with $b = 0$ in the following Lemma.

\begin{lemma} \label{lem-xy} For random variables $X,Y$ with cumulative distribution functions $F_X, F_Y$ respectively and $b,c \in \BR$, the following are equivalent:

i)  There exist a coupling such that $\p(X \le Y+b) =1$ and a coupling such that $\p(Y \le X+c) = 1$

ii)  For all $t,  F_Y(t-b) \le F_X(t) \le F_Y(t+c)$

iii)  There exists a coupling such that $\p(X-b \le Y \le X+c)=1$

iv) There exists a coupling such that $X-b \le Y \le X+c$ for all $\omega$.
\end{lemma}

\begin{proof}  To prove that i) implies ii) we use the coupling in which $\p(X \le Y+b) =1$ and calculate
     $$
     F_Y(t-b) = \p(Y \le t-b) \le \p(X \le t) + \p(X > Y+b) = \p(X \le t) = F_X(t)
     $$
and similarly for $F_X(t) \le F_Y(t+c)$.

To show ii) implies iv) we may use the quantile transform to couple $X$ and $Y$ to a single random variable $U$, uniformly distributed in (0,1).  This transform is written informally as
$$
    X(\omega) := F_X^{(-1)}(U(\omega)), \  Y(\omega) := F_Y^{(-1)}(U(\omega)) $$
and more formally as $X(\omega) := \inf \{t:  F_X(t) \ge U(\omega) \}$, $Y(\omega) := \inf \{t:  F_Y(t) \ge U(\omega) \}$.
Under this coupling, $F_Y(t-b) \le F_X(t)$ for all $t$ implies $X(\omega) \le Y(\omega)+b$ for all $\omega$, and $F_X(t) \le F_Y(t+c)$ for all $t$ implies $Y(\omega) \le X(\omega)+c$ for all $\omega$.  Together ii) implies $ X(\omega)-b \le Y(\omega) \le X(\omega)+c $ holds for all $\omega$.

   Finally, iv) implies iii) trivially, and  iii) implies i) a fortiori.
\qed \end{proof}

\subsection{A sandwich principle}

\begin{corollary} {\bf Sandwich principle}. \label{cor sandwich}  Suppose $X,Y,Z$ are random variables with cumulative distribution functions $F_X, F_Y, F_Z$ satisfying $F_Z(t) \le F_Y(t) \le F_X(t)$ for all $t$.  If there is a $c$-bounded coupling of $X$ and $Z$, as defined via \eqref{c bounded}, then there exists a $c$-bounded monotone coupling of $X$ and $Y$, as defined by \eqref{c bounded monotone}.
\end{corollary}

\begin{proof}  Using Lemma \ref{lem-xy} with $Z$ in place of $Y$, the $c$-bounded coupling of $X$ and $Z$ implies $F_X(t) \le F_Z(t+c)$ for all $t$.  Then $F_Y(t) \le F_X(t) \le F_Z(t+c) \le F_Y(t+c)$ for all $t$, and Lemma \ref{lem-xy} gives the existence of the $c$-bounded monotone coupling of $X$ and $Y$.
\qed \end{proof}

\paragraph{\bf Application}
As an illustration of the sandwich principle, we prove the following corollary.  The special case where
$X$ is Binomial$(m,p)$ for $p \in (0,1)$ and $m \ge 1$ was proved in \cite[Lemmas 3.2, 3.3]{GP}
by an explicit calculation.  The sandwich principle enables a short proof for the general case, with no calculation.

\begin{corollary}\label{cor condition sandwich}
Suppose $X$ satisfies {\bf BSBC}, with $c=1$, and the distribution of $Y$ is defined to be that of $X$, conditional on $X > 0$.
Then there is a coupling in which $Y-X \in [0,1]$ for all $\omega$.
\end{corollary}
\begin{proof}
Trivially, $Y$ stochastically dominates $X$.  Take the distribution of $Z$ to be the size biased distribution of $X$.  By assumption, there is a coupling in which $\p(Z-X \in [0,1])=1$.
Trivially, the distribution of  $Z$, initially defined as the size-biased distribution of $X$, is also the size-biased distribution of $Y$. Hence $Z$ dominates $Y$, and the sandwich principle applies with $c = 1$.
\qed \end{proof}

This result may be applied to any of the random variables $X$ discussed in Section \ref{sect examples}, and in particular to those obtained using Proposition \ref{prop indep sum}.  If $X$ is (non-negative) integer valued, for example, a Binomial random variable or a sum of independent Bernoulli random variables, then $Y$ is also integer valued and the conclusion of Corollary \ref{cor condition sandwich} is easily strengthened to $Y-X \in\{0,1\}$ for all $\omega$.

\section{Analysis of the upper tail and lower tail bounds} \label{sect upper and lower}

In this section we obtain results enabling us to compare the upper and lower tail bounds in Theorem \ref{theorem gamma} with those in Theorem \ref{theorem AB}.

\begin{lemma}   For $0 < u \le v$
       \begin{equation} \label{log gamma}
          \frac{\Gamma(v+1/2)}{\Gamma(u+1/2)} \ge \frac{v^v}{e^{v-u}u^u}.
        \end{equation}
\end{lemma}

\begin{proof}
For $x > 0$, using Gauss' formula, see \cite[Sect 12.3]{whittaker},
    $$
   ( \log \Gamma)'(x+ 1/2) = \int_0^\infty \left(\frac{e^{-t}}{t} - \frac{e^{-t(x+1/2)}}{1-e^{-t}}\right)\,dt
    $$
together with
   $$
        \log x = \int_0^\infty \left( \frac{e^{-t} - e^{-tx}}{t}\right)\,dt
        $$
we obtain
   \begin{eqnarray*}
   ( \log \Gamma)'(x+1/2) -  \log x
     & = & \int_0^\infty \left(\frac{e^{-tx}}{t} - \frac{e^{-t(x+1/2)}}{1-e^{-t}}\right)\,dt\\
     & = & \int_0^\infty \frac{e^{-t(x+1/2)}(2 \sinh (t/2) -t)}{t(1-e^{-t})}\,dt \\
     & \ge & 0.
   \end{eqnarray*}
Therefore for $0 < u \le v$,
  \begin{eqnarray*}
\frac{\Gamma(u+1/2)}{\Gamma(v+1/2)}
    & \ge & \exp\left\{ \int_u^v \log x \,dx \right\} = \exp\left\{ \left[\rule{0cm}{2.5ex}x \log x - x\right]_u^v\right\} = \frac{v^v}{e^{v-u}u^u}.
 \end{eqnarray*}
 \qed \end{proof}

\begin{proposition} \label{prop exponential} For $0 < a \le x$,
\begin{eqnarray}
\frac{a^{x-a} \ \Gamma(a+1)}{\Gamma(x+1)} & \le & \frac{\left(a+\frac{1}{2}\right)^{a+1/2}(ae)^{x-a}}{\left(x+\frac{1}{2}\right)^{x+1/2}} \label{exponential 1}\\
    & \le & \left(\frac{a+\frac{1}{2}}{x+\frac{1}{2}}\right)^{1/2} \left(\frac{a}{x}\right)^x e^{x-a}. \label{exponential 2}
\end{eqnarray}
\end{proposition}

\begin{proof}  Taking $u = a+1/2$ and $v=x+1/2$ in \eqref{log gamma}, we get
  $$ \frac{\Gamma(x+1)}{\Gamma(a+1)} \ge \frac{(x+\frac{1}{2})^{x+1/2}}{e^{x-a}(a+\frac{1}{2})^{a+1/2}},
  $$
giving \eqref{exponential 1}.   A simple calculus argument shows that $t \mapsto  (a+t)^a/(x+t)^x$ is decreasing for $t \in[0,\infty)$.  In particular $\left(a+\frac{1}{2}\right)^a/\left(x+\frac{1}{2}\right)^x \le a^a/x^x$, and this gives \eqref{exponential 2}.
 \qed \end{proof}

\begin{proposition} \label{prop lower exponential} For $0 < x <a$,
\begin{eqnarray}
\frac{\Gamma(a+1)}{a^{a-x} \Gamma(x+1)} & \ge & \frac{\left(a+\frac{1}{2}\right)^{a+1/2}}{(ae)^{a-x}\left(x+\frac{1}{2}\right)^{x+1/2}} \label{lower exponential 1}\\
    & \ge & \left(\frac{a+\frac{1}{2}}{x+\frac{1}{2}}\right)^{1/2} \left(\frac{a}{x}\right)^x e^{x-a}. \label{lower exponential 2}
\end{eqnarray}
\end{proposition}

\begin{proof}  Essentially the same as for Proposition \ref{prop exponential}, but with the roles of $x$ and $a$ switched.
 \qed \end{proof}

\section{Comparison of the bounds} \label{sect comparison}

\subsection{Relative strength of our upper and lower tail bounds} Here we compare the elementary ``product bounds'' (\ref{upper},\,\ref{lower}) given in Theorem \ref{simple thm} with the Gamma function bounds (\ref{gamma upper 1},\,\ref{gamma lower 1}) given in Theorem \ref{theorem gamma} and the bounds (\ref{AB upper},\,\ref{AB lower}) obtained in Theorem \ref{theorem AB} from the moment generating function estimate.

It is clear from Theorem \ref{theorem gamma} (together with the scaling relationship \eqref{scaling}) that the product bounds (\ref{upper},\,\ref{lower}) are at least as strong (i.e. small) as the corresponding Gamma bounds (\ref{gamma upper 1},\,\ref{gamma lower 1}).  However the relationship between the Gamma bounds and the bounds  (\ref{AB upper},\,\ref{AB lower}) in Theorem \ref{theorem AB} is more complicated.  See Propositions \ref{prop exponential} and \ref{prop lower exponential}.  For the upper tail, with $a \le x$, the Gamma bound \eqref{gamma upper 1} is sharper than the bound \eqref{AB upper} from Theorem \ref{theorem AB} by a factor $\sqrt{(a+c/2)/(x+c/2)}$.  However, the situation is reversed for the lower tail, with $0 \le x \le a$, where now the bound \eqref{AB lower} from Theorem \ref{theorem AB} is sharper than the Gamma bound \eqref{gamma lower 1} by a factor $\sqrt{(x+c/2)/(a+c/2)}$.  Since the Gamma bound \eqref{gamma lower 1} and the product bound \eqref{lower} agree whenever $a-x$ is an integer, this suggests (but does not prove) that the bound \eqref{AB lower} is the best of the three for the lower tail.

Numerical investigations suggest that \eqref{AB lower} is in fact the best estimate to use for the lower tail, and beats the simple product rule $\ell(x,a,c)$ of \eqref{lower}.  Recall however, from Remark \ref{rem ell j},  the lower tail bounds derived from
\eqref{lower one step} in combination with the one-sided Chebyshev inequality \eqref{one sided},  in particular the functions $\ell_j(x,a,c)$ defined in \eqref{lower j}.  Numerical investigations suggest that for all $a,c > 0$ with sufficiently large $a/c$ there exists $x \in (0,a)$ such that the bound given in \eqref{AB lower} is less than the product bound \eqref{lower} and is less than the one-sided Chebyshev estimate \eqref{one sided}, but is \emph{greater} than $\ell_j(x,a,c)$ for some nonnegative
integer $j\le (a-x)/c$.

\subsection{Hoeffding bounds} \label{sect Hoeffding} Suppose $X = X_1+ \cdots + X_n$ where the $X_i$ are independent and take values in $[0,1]$, and let $a = \e X$.  Clearly $a < \infty$ and to avoid trivialities we assume $a > 0$.  Hoeffding \cite[Thm 1]{hoeffding} proved that for $a \le x < n$
    $$\mathbb{P}(X \ge x) \le \left(\frac{a}{x}\right)^x \left(\frac{n-a}{n-x}\right)^{n-x}.
    $$
The inequality
     $$
     \left(\frac{n-a}{n-x}\right)^{n-x} \le e^{x-a}
     $$
for $0 < x < n$, (see for example \cite{hagerup}), and the fact that $\mathbb{P}(X \ge x) = 0$ for $x > n$, together give the upper tail bound
    \begin{equation} \label{hoeffding upper tail}
    \mathbb{P}(X \ge x) \le \left(\frac{a}{x}\right)^x e^{x-a} \quad \mbox{ for all } x \ge a.
    \end{equation}
A similar argument with $X_i$ replaced by $1-X_i$ and $X$ replaced by $n-X$ gives the lower tail bound
    \begin{equation} \label{hoeffding lower tail}
    \mathbb{P}(X \le x) \le \left(\frac{a}{x}\right)^x e^{x-a} \quad \mbox{ whenever } 0 \le x \le a.
    \end{equation}
Notice that the right sides of \eqref{hoeffding upper tail} and \eqref{hoeffding lower tail} do not depend on the number $n$ of summands in $X$.  Other related inequalities, also referred to as Hoeffding or Chernoff--Hoeffding bounds, involve the parameter $n$.

From Proposition \ref{prop indep sum} we know that any random variable $X$ of the form above admits a 1-bounded size bias coupling.  Therefore the Hoeffding bounds (\ref{hoeffding upper tail},\,\ref{hoeffding lower tail}) are a special case of the bounds (\ref{AB upper},\,\ref{AB lower}) in our Theorem \ref{theorem AB}.  Our best upper tail bound, given by \eqref{upper} is smaller than the Hoeffding upper tail bound \eqref{hoeffding upper tail} by a factor $\displaystyle{\left(\frac{a+1/2}{x+1/2}\right)^{1/2}}$.  Moreover our results have a broader scope, applying to any random variable $X$ which admits a 1-bounded size bias coupling.  In particular it applies to sums of independent nonnegative random variables such as Uniform[0,4] and L\'evy[0,1] (as discussed in Examples \ref{example uniform} and \ref{example inf div}).

\bibliographystyle{spmpsci}      
\bibliography{concbib}   

%
%

\end{document}